\documentclass{article}
\usepackage[utf8]{inputenc}

\usepackage{geometry}

\usepackage{float}
\usepackage{caption}
\usepackage{graphicx} 

\usepackage{amsthm}
\usepackage{amsmath}
\usepackage{amssymb}
\usepackage{mathrsfs}
\usepackage{amsfonts}
\usepackage{hyperref}
\usepackage{lineno}
\usepackage{fancyhdr}
\usepackage{dsfont}
\usepackage{indentfirst}
\usepackage{url}
\usepackage{cleveref}
\usepackage{enumerate}

\newtheorem{theorem}{Theorem}[section]
\newtheorem{corollary}[theorem]{Corollary}
\newtheorem{lemma}[theorem]{Lemma}

\newtheorem{proposition}[theorem]{Proposition}
\newtheorem{conjecture}[theorem]{Conjecture}
\newtheorem{conjecturescheme}[theorem]{Conjecture Scheme}

\newtheorem{question}[theorem]{Question}

\theoremstyle{definition}
\newtheorem{remark}[theorem]{Remark}
\newtheorem*{remark*}{Remark}
\newtheorem{definition}[theorem]{Definition}

\newcommand{\dchi}{\vec\chi}
\renewcommand{\-}{\setminus}
\newcommand{\N}{\mathbb N}
\newcommand{\Z}{\mathbb Z}
\newcommand{\rni}{\ref{lrni}}
\newcommand{\rnii}{\ref{lrnii}}
\newcommand{\rniii}{\ref{lrniii}}
\newcommand{\rniv}{\ref{lrniv}}
\newcommand{\cti}{Conjecture~\ref{conjbag}\rni}
\newcommand{\ctii}{Conjecture~\ref{conjbag}\rnii}
\newcommand{\ctiii}{Conjecture~\ref{conjbag}\rniii}
\newcommand{\ctiv}{Conjecture~\ref{conjbag}\rniv}
\newcommand{\kpi}{Theorem~\ref{kpall}\rni}
\newcommand{\kpii}{Theorem~\ref{kpall}\rnii}
\newcommand{\kpiii}{Theorem~\ref{kpall}\rniii}
\newcommand{\eq}[1]{\begin{align*}#1\end{align*}}
\newcommand{\DT}{C_3}
\DeclareMathOperator{\dist}{dist}
\DeclareMathOperator{\kp}{kp}
\usepackage{subfigure}
\usepackage{tikz}
\usetikzlibrary{decorations.markings}
\usetikzlibrary{backgrounds}

\usetikzlibrary{shapes, arrows, calc, arrows.meta, fit, positioning} 
\tikzset{  
	-stealth,auto,node distance =0.8 cm and 1 cm, thin, 
	state/.style ={circle, draw, inner sep=0.2pt}, 
	point/.style = {circle, draw, inner sep=0.18cm, fill, node contents={}},  
	el/.style = {inner sep=2pt, align=right, sloped}  
}

\title{A variable version of the quasi-kernel conjecture}

\author{Jiangdong Ai\thanks{School of Mathematical Sciences and LPMC, Nankai University, Tianjin 300071, P.R.
China. Email: jd@nankai.edu.cn. Partially supported
by the Fundamental Research Funds for the Central Universities, Nankai University.}, Xiangzhou Liu\thanks{Department of Mathematics, Tiangong University, Tianjin 300071, P.R.
China. Email: i19991210@163.com.}, Fei Peng\thanks{Department of Mathematics, National University of Singapore, Singapore 119076. Email: pfpf@u.nus.edu.}}

\begin{document}

	\maketitle
 \begin{abstract}
     A quasi-kernel of a digraph $D$ is an independent set $Q$ such that every vertex can reach~$Q$ in at most two steps. A 48-year conjecture made by P.L. Erd\H{o}s and Sz\'ekely, denoted the \textit{small QK conjecture}, says that every sink-free digraph contains a quasi-kernel of size at most~$n/2$. 
     Recently, Spiro posed the \textit{large QK conjecture}, that every sink-free digraph contains a quasi-kernel~$Q$ such that $|N^-[Q]|\geq n/2$, and showed that it follows from the small QK conjecture. 

    In this paper,  we establish that the large QK conjecture implies the small QK conjecture with a weaker constant. We also show that the large QK conjecture is equivalent to a sharp version of it, answering affirmatively a question of Spiro. We formulate variable versions of these conjectures, which are still open in general.
    
    Not many digraphs are known to have quasi-kernels of size $(1-\alpha)n$ or less. We show this for digraphs with bounded dichromatic number, by proving the stronger statement that every sink-free digraph contains a quasi-kernel of size at most $(1-1/k)n$, where $k$ is the digraph's \textit{kernel-perfect number}.
 \end{abstract}

\section{Introduction}
 We refer readers to \cite{bang2008digraphs} for the standard terminology and notation not introduced in this paper.
Let $D=(V(D), A(D))$ be a digraph. If $xy\in A(D)$, we say that $y$ is an out-neighbor of $x$, and $x$ is an in-neighbor of $y$. Let $v \in V(D)$. The open (or closed) out-neighborhood  (or in-neighborhood) of $v$ in $D$ is defined as follows. (The subscript $_D$ is omitted if the underlying digraph is clear.)
\eq{N_{D}^{+}(v)&=\{u\in V(D):vu\in A(D)\},\ N_{D}^{+}[v]= N_{D}^{+}(v)\cup \{v\},\\
N_{D}^{-}(v)&=\{u\in V(D):uv\in A(D)\},\ N_{D}^{-}[v]= N_{D}^{-}(v)\cup \{v\}.}
Given vertices $u,v$ of a digraph $D$, let $\dist(u,v)\in\Z^{\ge0}\cup\{\infty\}$ denote the length of a shortest directed path from $u$ to $v$. With a set of vertices $S$, denote $\dist(u,S) = \min_{v\in S} \dist(u,v)$, and analogously $\dist(S,v) = \min_{u\in S} \dist(u,v)$. Define \eq{N^+(S)&=\{v\in V(D):\dist(S,v)=1\},\ N^+[S]=\{v\in V(D):\dist(S,v)\le1\},\\N^-(S)&=\{u\in V(D):\dist(u,S)=1\},\ N^-[S]=\{u\in V(D):\dist(u,S)\le1\},} coinciding with the earlier definitions if $S$ is a singleton. As a preview of what's to come, define $$N^{--}(S)=\{u\in V(D): \dist(u,S) =2\},\ N^{--}[S]=\{u\in V(D):\dist(u,S)\le 2\}.$$
We call an independent set $K\subseteq V(D)$ a \emph{kernel} of $D$ if $N^-[K]=V(D)$, namely each vertex not in~$K$ has an out-neighbor in $K$. Not every digraph has a kernel (consider an odd dicycle), although every digraph without odd dicycles has one \cite{richardson1946weakly}. 
Chv\'{a}tal and  Lov\'{a}sz \cite{MR0414412} introduced the notion of quasi-kernels. An independent set $Q\subseteq V(D)$ is said to be a \emph{quasi-kernel} of $D$ if $N^{--}[Q]=V(D)$. 
Notably, \cite{MR0414412} proves that every digraph has a quasi-kernel. 
Thus, it is natural to ask if one can always find a quasi-kernel that is small (or large). Since all the quasi-kernels in a tournament must be singletons, asking for a large quasi-kernel is not as interesting as asking for a small quasi-kernel. P.L. Erd\H{o}s and Sz\'{e}kely made the following conjecture on the existence of small quasi-kernels. 

\begin{conjecture}[Small Quasi-kernel Conjecture \cite{erdHos2010two}, 1976]\label{small}
If $D$ is a sink-free digraph, then $D$ has a quasi-kernel $Q$ with $|Q|\leq \frac{1}{2}|V(D)|$.
\end{conjecture}

Here, a digraph is said to be \textit{sink-free} if it has no sinks, where \textit{sinks} denote the vertices without in-neighbors. \textit{Sources} and \textit{source-free} digraphs are defined analogously, referring to the out-neighbors.
The sink-free condition cannot be removed, as can be seen by considering a digraph with many sinks. The constant $1/2$ is the best possible, as can be seen by considering (disjoint unions of) directed 2-cycles and 4-cycles. 

Conjecture~\ref{small} is wide open: the best bound that works for all sink-free digraphs appears to be $|Q|\le n-\sqrt n$, where $n$ is the number of vertices \cite{spiro2024generalized}. However, there have been substantial results that confirm that Conjecture~\ref{small} holds on certain classes of digraphs. 
Heard and Huang \cite{heard2008disjoint}  showed that each sink-free digraph $D$ has two disjoint quasi-kernels if $D$ is semicomplete multipartite, quasi-transitive, or locally semicomplete. As a consequence, Conjecture ~\ref{small} is true for these three classes of digraphs. Van Hulst \cite{Hulst2021KernelsAS} showed that Conjecture~\ref{small} holds for all digraphs containing kernels. Kostochka, Luo and Shan \cite{MR4477848} proved that Conjecture~\ref{small} holds for digraphs with chromatic number at most~4. Ai et al. \cite{MR4569676} proved that Conjecture~\ref{small} holds for one-way split digraphs. 
We refer the interested reader to the nice survey by P.L. Erd\H{o}s et al.
\cite{erdHos2023small} for a more thorough overview of this problem.

Recently, Spiro introduced a way to ask for a large quasi-kernel in general: that is, to measure ``largeness" not by the size of $Q$ but by that of $N^-[Q]$. Note the removal of the sink-free condition below. 

\begin{conjecture}[Large quasi-kernel conjecture, \cite{spiro2024generalized}]\label{large}
Every digraph $D$ has a quasi-kernel $Q$ such that $|N^-[Q]|\ge \frac{1}{2}|V(D)|$.

\end{conjecture}

Interestingly, \cite{spiro2024generalized} shows that Conjecture~\ref{small} implies Conjecture~\ref{large}, and obtains several results on both conjectures. In this paper, we show that the converse is also true to some degree: namely, Conjecture~\ref{large} implies Conjecture~\ref{small} but with a weaker constant ($2/3$) instead of $1/2$.  We utilize the following conjecture scheme to enable more extended discussions:

\begin{conjecturescheme}\label{conjbag}
    Fix some $0<\alpha\le 1/2$. One can conjecture the following, for all digraphs $D$ on $n$ vertices:
    \begin{enumerate}[I.]
        \item \label{lrni} (Small quasi-kernel conjecture) If $D$ is sink-free, then it has a quasi-kernel with size at most $(1-\alpha)n$.
        \item \label{lrnii} (Small quasi-kernel conjecture with sources) $D$ has a quasi-kernel with size at most $n-\alpha s$, where $s$ is the number of sources in $D$ that are not sinks.
        \item \label{lrniii} (Large quasi-kernel conjecture) $D$ has a quasi-kernel $Q$ such that $|N^-[Q]|\ge \alpha n$.
        \item \label{lrniv} (Sharp large quasi-kernel conjecture) $D$ has a quasi-kernel $Q$ such that $|Q|/2+|N^-(Q)|\ge \alpha n$.
    \end{enumerate}
\end{conjecturescheme}

Our next result will focus on the first three statements. At a glance, the only obvious relationship among them is that {\cti} implies a ``sink-free version" of {\ctii} (i.e., assuming additionally that $D$ is sink-free). It is shown in \cite{spiro2024generalized} that {\cti} implies {\ctiii}. Here we provide a clearer picture.

\begin{proposition}\label{equiv}
    {\ctii} and {\ctiii} are equivalent, and equivalent respectively to their sink-free version. {\cti} implies {\ctiii}, and {\ctiii}($\alpha$) implies {\cti}($\frac\alpha{1+\alpha}$).
\end{proposition}

Since a general bound of the form $(1-\Theta(1))n$ is not known for Conjecture~\ref{small}, Proposition~\ref{equiv} suggests that the large quasi-kernel conjecture is a safe but also effective target to work on: proving {\ctiii}, for any $\alpha$, would imply a breakthrough on the small quasi-kernel conjecture. It would be interesting to know whether {\ctiii} is completely equivalent to {\cti}. Of course, a negative answer would disprove Conjecture~\ref{small}.
\begin{question}\label{q31equiv}
    Is {\ctiii} equivalent to {\cti}? That is, for all $0<\alpha\le 1/2$, if all digraphs on $n$ vertices have a quasi-kernel $Q$ with $|N^-[Q]|\ge\alpha n$, then all sink-free digraphs on $n$ vertices have a quasi-kernel with size at most $(1-\alpha)n$.
\end{question}
Note that we are not trying to say that these conjectures are equivalent on the same digraph~$D$. In fact, one can see that statements {\rni} and {\rniii} cannot both be false on the same digraph.

As noted in \cite{spiro2024generalized}, {\ctiii}($\frac12$) is only asymptotically sharp. For this reason, we would like to advertise the ``sharp large quasi-kernel conjecture": every disjoint union of Eulerian tournaments of arbitrary sizes witnesses the sharpness (if true) of {\ctiv}($\frac12$). In the next result, we will show that this sharp version is indeed equivalent to {\ctiii}. Setting $\alpha=1/2$, this implies that \cite[Question~7.9]{spiro2024generalized} is equivalent to Conjecture~\ref{large}, answering affirmatively a question at the end of \cite[Section 7]{spiro2024generalized}. Moreover, we observe that if the quasi-kernel requirement in {\ctiv}($\frac12$) is relaxed to allow any independent set, it will be not only sharp but also true. This extends Lemma~4.1b of \cite{spiro2024generalized} to a tight result.

\begin{proposition}\label{aboutiv}~
    \begin{enumerate}[(a)]
        \item {\ctiii} and {\ctiv} are equivalent, and equivalent respectively to their sink-free version.
        \item Every digraph $D$ contains an independent set $I$ with $|I|/2+|N^-(I)|\ge n/2$.
    \end{enumerate}
\end{proposition}

As noted before, there are not many classes of digraphs on which {\cti} is known to hold for some $\alpha$. We recall the notion of \textit{kernel-perfect} digraphs, and introduce a digraph measure called the \textit{kernel-perfect number}.
\begin{definition}
    A digraph $D$ is said to be \textit{kernel-perfect} if every induced subdigraph of it has a kernel. In this paper, we conveniently call a vertex set $S\subseteq V(D)$ kernel-perfect if $D[S]$ is kernel-perfect. The \textit{kernel-perfect number} of a digraph $D$, denoted $\kp(D)$, is the smallest $k$ such that $V(D)$ can be partitioned into $k$ kernel-perfect subsets.
\end{definition}
We recall the related notions of \textit{chromatic number} and \textit{dichromatic number} of digraphs.
\begin{definition}
Let $D$ be a digraph. 
\begin{itemize}
    \item The chromatic number $\chi(D)$ is the smallest $k$ such that $V(D)$ can be partitioned into $k$ subsets, each of which induces an independent set.
    \item The dichromatic number $\dchi(D)$ is the smallest $k$ such that $V(D)$ can be partitioned into $k$ subsets, each of which induces an acyclic set.
\end{itemize}
\end{definition}
Note that $\kp(D)\le\dchi(D)\le\chi(D)$, since independent sets are acyclic and acyclic sets are kernel-perfect. Also, $\kp(D)\le\lceil\chi(D)/2\rceil$ because one can group the color classes of the underlying graph two by two, so that each group is odd-dicycle-free, hence kernel-perfect~\cite{richardson1946weakly}. A main result of~\cite{MR4477848} suggests that Conjecture~\ref{small} holds on digraphs with kernel-perfect number at most 2, which include all digraphs with dichromatic number at most 2 or chromatic number at most 4 (hence all planar digraphs).
We extend this to a variable version that applies to digraphs with any bounded kernel-perfect number, and in addition, prove the large quasi-kernel analog.

\begin{theorem}\label{kpall}
    Let $D$ be a digraph and $k=\max(\kp(D),2)$. Then {\cti}($\frac1k$), {\ctii}($\frac1k$) and {\ctiii}($\frac1k$) hold on $D$.
\end{theorem}
\begin{corollary}
    Let $D$ be a digraph with $\dchi(D)\le k$ or $\chi(D)\le 2k$, where $k\ge 2$. Then {\cti}($\frac1k$), {\ctii}($\frac1k$) and {\ctiii}($\frac1k$) hold on $D$.
\end{corollary}
\begin{remark}~
\begin{itemize}
    \item When $\kp(D)\le 2$, {\kpi} is \cite[Theorem~2]{MR4477848}, except that {\kpi} makes no claims when $D$ has sinks. With more care, however, an analogous claim can be proven similarly.

    \item When $\kp(D)=1$ (i.e., $D$ is kernel-perfect and not the null digraph), and one ignores the $\alpha\le1/2$ condition in the conjecture scheme, {\cti}(1) and {\ctiv}(1) do not hold on~$D$, while {\ctii}(1) and {\ctiii}(1) still hold on~$D$.

    \item Theorem~\ref{kpall} is perhaps another indication that Question~\ref{q31equiv} probably has a positive answer.
\end{itemize}
    
\end{remark}

We end this section with two more questions:
\begin{question}\label{inciv}
    Can Theorem~\ref{kpall} cover {\ctiv} as well? That is, all digraphs $D$ on $n$ vertices have a quasi-kernel $Q$ such that $|Q|/2+|N^-(Q)|\ge n/\max(\kp(D),2)$.
\end{question}
A negative answer to Question~\ref{inciv} would disprove Conjectures~\ref{large} and~\ref{small} by Propositions~\ref{aboutiv} and~\ref{equiv}. We also ask for general upper bounds of the kernel-perfect number.
\begin{question}
    Which graphs have the largest $\kp(D)$? Can this number be linear w.r.t $n$?
\end{question}
We observe that the iterative blowup of $C_3$ has $\kp(C_3^{\odot k})=\lceil 1.5\kp(C_3^{\odot k-1})\rceil =\Theta\left(n^{\log_3(1.5)}\right)$.
\section{Equivalent formulations}

Now we prove Proposition~\ref{equiv}. In the proof, the Roman numerals refer to the corresponding conjecture in Conjecture Scheme~\ref{conjbag}, and the ``sf" suffix denotes the corresponding sink-free version.
\begin{proof}[Proof of Proposition~\ref{equiv}]
    The directions {\rni} $\to$ {\rnii}sf $\to$ {\rniii}sf $\to$ {\rniii} are implicitly shown and used in \cite[Proposition~2.7]{spiro2024generalized}; we prove them for completeness. In addition, we show that {\rniii} $\to$ {\rnii} $\to$ {\rni}($\frac\alpha{1+\alpha}$).
    \paragraph{{\rni} $\to$ {\rnii}sf} This is clear: assuming $D$ is sink-free, one can simply apply {\cti} on $D$.
    \paragraph{{\rnii}sf $\to$ {\rniii}sf} Let $D$ be a sink-free digraph on $n$ vertices. Fix $C\in\N$, to be chosen later. Construct a digraph $D'$ by keeping $D$ and add, for each $v\in D$, $C$ new vertices pointing an arc towards $v$. Note that all these $Cn$ new vertices are sources but not sinks in $D'$, so $D'$ is still sink-free. Assuming {\ctii}sf, we obtain a quasi-kernel $Q'$ of $D'$ with size at most $(C+1)n-\alpha Cn$. Note that $Q=Q'\cap V(D)$ is a quasi-kernel of $D$. Moreover, for each $v\notin N^-_D[Q]$, all the new vertices pointing to $v$ must be included in $Q'$. Thus, $$C(n-|N^-_D[Q]|)\le|Q'|\le(C+1)n-\alpha Cn.$$ It follows that $|N^-_D[Q]|\ge \alpha n-n/C$. We are done because $C$ can be made arbitrarily large.
    \paragraph{{\rniii}sf $\to$ {\rniii}} 
    Suppose $D$ is a minimal counterexample for {\ctiii}. Assuming {\ctiii}sf, $D$ must have a sink $v$. Let $A=V(D)\-N^-[v]$. Applying {\ctiii} on the smaller digraph $D[A]$, we obtain a quasi-kernel $Q_A$ of $D[A]$ with $|N^-_{D[A]}[Q_A]|\ge \alpha |A|$. Note that $Q=Q_A\cup\{v\}$ is a quasi-kernel of $D$ with \eq{|N^-_D[Q]|&=|N^-_{D[A]}[Q_A]\sqcup N^-_D[v]|\ge\alpha|A|+(n-|A|)\ge\alpha n,} so it works. This is a contradiction.
    \paragraph{{\rniii} $\to$ {\rnii}} Let $D$ be a digraph on $n$ vertices. Denote $S$ the set of sources in $D$ that are not sinks, and $A=V(D)\-S$. Without loss of generality, we can assume each vertex in $S$ has exactly one out-neighbor (which must be in $A$), since removing an arc from $S$ to $A$ can only make it harder to find a small quasi-kernel. Let $s=|S|$ and $t=|A|$. Fix $C\in\N$, to be chosen later. For each $a\in A$, let $$n_a=C|N^-(a)\cap S|+1.$$ Construct a digraph $B$ based on $D[A]$ by replacing each $a\in A$ with $n_a$ copies of $a$ (vertices in~$S$ are discarded). In some sense, $B$ is a weighted blowup of $D[A]$. Note that every maximal quasi-kernel of $B$ naturally induces a maximal quasi-kernel of $D[A]$, and that $|V(B)|=Cs+t$. Let~$Q_B$ be a quasi-kernel of $B$ that maximizes $|N^-[Q_B]|$, and $Q_A$ be the induced maximal quasi-kernel of $D[A]$. Denote $A'$ the set of vertices in $A$ whose copies in $B$ are not in $N^-[Q_B]$ (these copies are either all in or all not in because $Q_B$ is maximal). Note that $$Q_A\cup\bigsqcup_{a\in A'}\left(N^-(a)\cap S\right)$$ is a quasi-kernel of $D$, and that its size is \begin{align*}&|Q_A|+\sum_{a\in A'}\frac{n_a-1}C\\=\ &|Q_A|+\frac1C(|V(B)\-N^-[Q_B]|-|A'|)\\\le\ &|A|+\frac{1-\alpha}{C}|V(B)|\tag{assuming {\ctiii}}\\=\ &\left(1+\frac{1-\alpha}{C}\right)t+(1-\alpha)s.\end{align*} Since $C$ can be made arbitrarily large, there is a quasi-kernel of $D$ with size at most $n-\alpha s$.
    \paragraph{{\rnii} $\to$ {\rni}($\frac\alpha{1+\alpha}$)} Let $D$ be a sink-free digraph. We claim that it has a quasi-kernel with size at most $n/(1+\alpha)$, where $\alpha$ is such that {\ctii}($\alpha$) holds. Let $Q$ be a minimal quasi-kernel of $D$, $N=N^-(Q)$ and $M=N^{--}(Q)=V(D)\-(Q\cup N)$. Take a maximal directed matching from $N$ to $Q$. Denote $Q_1$ the set of vertices in $Q$ it covers, and $Q_2=Q\-Q_1$. By the maximality of the matching, every vertex in $N$ has an out-neighbor in $Q_1$. Thus, by the minimality of $Q$, every vertex in $Q_2$ has no out-neighbor in $N$, so its out-neighbors are all in $M$. Note that $Q_2$ is a set of sources that are not sinks in $D[Q_2\cup M]$. Let $r=|Q_1|$, $s=|Q_2|$, $p=|N|\ge r$, and $m=|M|$. Assuming {\ctii}, there is a quasi-kernel $Q'$ of $D[Q_2\cup M]$ with size at most $m+(1-\alpha)s$. Note that $Q'\cup (Q_1\-N^-(Q'))$ is a quasi-kernel of $D$, with size at most $r+m+(1-\alpha)s$. This quantity and $|Q|=r+s$ cannot be both greater than $n/(1+\alpha)$: otherwise, \eq{n&<(1+\alpha)(r+m+(1-\alpha)s)\\&\le (1+\alpha)(p+m+(1-\alpha)s)\\&= (1-\alpha^2)(p+m)+(\alpha+\alpha^2)(n-(r+s))+(1+\alpha)(1-\alpha)s\\&< (1-\alpha^2)(p+m)+(\alpha+\alpha^2)(n-n/(1+\alpha))+(1+\alpha)(1-\alpha)s\\&\le n,} a contradiction. Hence $D$ has a quasi-kernel with size at most $n/(1+\alpha)$.
\end{proof}

Reusing a few of the earlier techniques, we next prove Proposition~\ref{aboutiv}.
\begin{proof}[Proof of Proposition~\ref{aboutiv}]
    As {\rniv} $\to$ {\rniii} is trivial, we show {\rniv}sf $\to$ {\rniv}, {\rniii} $\to$ {\rniv}, and then part (b).
    \paragraph{{\rniv}sf $\to$ {\rniv}} We use the proof idea of {\rniii}sf $\to$ {\rniii}. Suppose $D$ is a minimal counterexample for {\ctiv}. Assuming {\ctiv}sf, $D$ must have a sink $v$. Let $A=V(D)\-N^-[v]$. Applying {\ctiv} on the smaller digraph $D[A]$, we obtain a quasi-kernel $Q_A$ of $D[A]$ with $|Q_A|/2+|N^-_{D[A]}(Q_A)|\ge \alpha |A|$. Note that $Q=Q_A\cup\{v\}$ is a quasi-kernel of $D$ with  \eq{|Q|/2+|N^-_D(Q)|&=|Q_A\cup\{v\}|/2+|N^-_{D[A]}(Q_A)\sqcup N^-_D(v)|\\&\ge(|Q_A|+1)/2+(\alpha|A|-|Q_A|/2)+(n-|A|-1)\\&=(\alpha-1)|A|-1/2+n\\&\ge(\alpha-1)(n-1)-1/2+n\\&\ge \alpha n,} so it works. This is a contradiction.
    \paragraph{{\rniii} $\to$ {\rniv}} 
    
    Let $\alpha$ be such that {\ctiii} holds. Consider \eq{S=\{\beta\in[0,1]\ |\ \forall \text{ digraph } D,\ \exists\text{ quasi-kernel }Q:\beta|Q|+|N^-(Q)|\ge\alpha|V(D)|\},} and observe that $S$ is a closed interval containing 1. So let $b=\min(S)$. We are done if $1/2\in S$, so assume $b>1/2$. We claim that $(b+1)/3\in S$, which would contradict with $b=\min(S)$. 
    
    Let $\DT$ denote the directed triangle. For any digraph $D$, denote $D'$ the \textit{$\DT$-blowup} of $D$, which is the digraph obtained by replacing each vertex in $D$ by a copy of $\DT$, so that the arcs between different copies are as induced by $D$. For $v\in D$, denote $f(v)$ the set of vertices that take the place of $v$ in~$D'$ (so $D'[f(v)]\simeq \DT$). Since $b\in S$, there is a quasi-kernel $Q'$ of $D'$ such that $b|Q'|+|N^-(Q')|\ge\alpha|V(D')|$. Its projection onto~$D$, \eq{Q=\{v\in D:f(v)\cap Q'\neq\emptyset\},} is a quasi-kernel of $D$. For all $v\in Q$, let $g(v)=f(v)\cap Q'$, containing exactly one vertex of $D'[f(v)]$. Note that \eq{Q'&=\bigcup_{v\in Q} g(v),\\N^-(Q')&=\bigcup_{v\in Q}N^-_{D'[f(v)]}(g(v))\ \ \cup\ \bigcup_{v\in N^-(Q)}f(v),} where all the unions are disjoint. Thus, \eq{\alpha |V(D')|&\le b|Q'|+|N^-(Q')|\\&=b|Q|+|Q|+3|N^-(Q)|\\&=\frac{V(D')}{V(D)}\left(\frac{b+1}3|Q|+|N^-(Q)|\right).} Since $D$ is arbitrary, this shows $(b+1)/3\in S$, a contradiction.
    \paragraph{Part (b)}
    We show that every digraph $D$ contains an independent set $I$ with $|I|/2+|N^-(I)|\ge n/2$. The proof in \cite[Lemma~4.1b]{spiro2024generalized} can be adapted for this. Here we simplify it slightly. We claim the slightly stronger statement that $D$ contains a maximal independent set $I$ with $|N^-(I)|\ge|N^+(I)|$. This would imply what we need because for every maximal independent set $I$, $N^-(I)\cup N^+(I)=V(D)\setminus I$.

    Suppose $D$ is a minimal counterexample to the claim, which must be nonempty. By the handshaking dilemma, there is $v\in D$ whose in-degree is no less than out-degree. By the minimality of $D$, the smaller, possibly empty digraph $D'=D[V(D)\setminus (N^+[v]\cup N^-(v))]$ contains a maximal independent set $I'$ with $|N^-_{D'}(I')|\ge |N^+_{D'}(I')|$. Note that $I=I'\cup\{v\}$ is a maximal independent set in $D$, satisfying \eq{|N^-(I)|&=|N^-(v)|+|N^-_{D'}(I')|\ge |N^+(v)|+|N^+_{D'}(I')|=|N^+(I)|,} so it works. This is a contradiction.
\end{proof}

\section{Digraphs with bounded kernel-perfect number}
In this section, we prove {\kpi}. We start with a convenient lemma.
As independent sets and acyclic sets are kernel-perfect, Lemma~\ref{kplem} directly implies \cite[Lemma~3.1]{spiro2024generalized}, \cite[Lemma~4.3]{spiro2024generalized} and \cite[Lemma~A.1]{spiro2024generalized}.
\begin{lemma}\label{kplem}
    Let $D$ be a digraph with a kernel-perfect set $P\subseteq V(D)$. Then $D$ has a quasi-kernel~$Q$ such that $P\subseteq N^-[Q]$ and $Q\cap N^-(P)=\emptyset$.
\end{lemma}
\begin{proof}
    Without loss of generality, assume $P$ is the maximal kernel-perfect set in $D[V(D)\-N^-(P)]$. Since adding a sink to a kernel-perfect set keeps it kernel-perfect, this means that $N^-[P]=V(D)$. Thus, an arbitrary kernel of $P$ works.
\end{proof}
Somewhat curiously, we need only the first property of $Q$ to prove {\kpiii} and {\kpii}, and only the second property to prove {\kpi}, not using the fact that they can be satisfied at the same time.
\begin{proof}[Proof of {\kpiii}]
    This is clear: Suppose $V(D)=V_1\cup\dots\cup V_k$ is the kernel-perfect partition. Apply Lemma~\ref{kplem} on the largest part to get a quasi-kernel $Q$ with $|N^-[Q]|\ge |V(D)|/k$.
\end{proof}
\begin{proof}[Proof of {\kpii}]
    We note that the proof of ``{\rniii} $\to$ {\rnii}" in Proposition~\ref{equiv} can also show that {\kpiii} implies {\kpii}. Say we want to show that {\ctii}($\frac1k$) holds on the digraph $D$. To use that reduction argument, we just need to be able to apply {\ctiii}($\frac1k$) on a digraph~$B$, which is a weighted blowup of some induced subdigraph of $D$ ($D[A]$). Thus, $\kp(B)=\kp(D[A])\le\kp(D)$, so if $D$ has a small kernel-perfect number, so does $B$. Hence {\kpiii} can be correctly invoked.
\end{proof}
Actually, the proof of ``{\rnii} $\to$ {\rni}($\frac\alpha{1+\alpha}$)" in Proposition~\ref{equiv} would also follow through with the kernel-perfect number condition, concluding a weaker bound than we need for {\kpi}. To get the precise bound, we use a standalone proof inspired by \cite[Theorem~2]{MR4477848}.
\begin{proof}[Proof of {\kpi}]
    Suppose $V(D)=V_1\cup\dots\cup V_k$ is a partition (empty sets allowed) such that each $D[V_i]$ is kernel-perfect. Without loss of generality, assume that $V_1$ is maximally kernel-perfect: in particular, every~${v\notin V_1}$ has an out-neighbor in $V_1$. Let~$K$ be a kernel of $D[V_1]$, and $K_0$ be a minimal subset of~$K$ such that $N^-(K_0)=N^-(K)$. Let $V'_1=N^-(K)\cup K_0$ and observe that $|K_0|\le|N^-(K)|$: for all~$v\in K_0$, because $K_0$ is minimal, there must be some $u=u(v)\in N^-(K)$ whose only out-neighbor in~$K_0$ is~$v$. For all $i\in\{2,\dots,k\}$, let $V'_i=V_i\-N^-(K)$. Then let $V'_0=K\-K_0$. Note that $\{V'_i\}_{i=0,\dots,k}$ is a partition of $V(D)$. 
    Let \eq{W=V(D)\-V'_1=V'_0\cup V'_2\cup\dots\cup V'_k.} If for some $i\in\{2,\dots,k\}$, $|N^-(V'_i)\cap V'_0|\ge|W|/k$, then since $V'_i$ is a kernel-perfect set in $D[W]$, by Lemma~\ref{kplem} there is a quasi-kernel $Q$ of $D[W]$ that is disjoint with $N^-_{D[W]}(V'_i)$. Observe that $Q\cup(K_0\-N^-(Q))$ is a quasi-kernel of $D$, and \eq{|Q\cup(K_0\-N^-(Q))|&\le |W\-N^-_{D[W]}(V'_i)|+|K_0|\\&\le |W|-|W|/k+|V'_1|/2\\&\le (k-1)n/k.} Otherwise, observe that $(N^-(W)\cap V'_0)\cup K_0$ is a quasi-kernel of $D$, and \eq{|(N^-(W)\cap V'_0)\cup K_0|&\le \sum_{i=2}^k|N^-(V'_i)\cap V'_0|+|K_0|\\&\le (k-1)|W|/k+|V'_1|/2\\&\le (k-1)n/k.} Either way, we find a quasi-kernel of $D$ with size at most $(k-1)n/k$.
\end{proof}
\bibliographystyle{plain}
\bibliography{refs} 

\end{document}